\newif\ifpictures
\picturestrue

\newif\ifcomment
\commenttrue

\documentclass[12pt]{amsart}
\usepackage{latex_base}
\usepackage{macros}

\author{Khazhgali Kozhasov* \protect\footnote{*C\MakeLowercase{orresponding author}.}and Jean Bernard Lasserre}\thanks{Research of the second author was funded by the European Research Council (ERC) under the European Union's Horizon 2020 research and innovation program (grant agreement ERC-ADG 666981 TAMING)}

\address{Khazhgali Kozhasov, Technische Universit\"at Braunschweig, Institut f\"ur Analysis und Algebra, Universit\"atsplatz 2, 38106 Braunschweig,
 Germany\medskip}

\email{k.kozhasov@tu-braunschweig.de}

\address{Jean Bernard Lasserre, LAAS-CNRS, Universit\'e de Toulouse, 7 avenue du colonel Roche, F-31400 Toulouse, France\medskip}

\email{lasserre@laas.fr}



\title[Nonnegative forms with sublevel sets of minimal volume]{Nonnegative forms with sublevel sets\\ of minimal volume}

\begin{document}

\maketitle

\begin{changemargin}{1.5cm}{1.5cm}
{\bf\noindent Abstract.}
We show that the Euclidean ball has the smallest volume among sublevel sets of nonnegative forms
of bounded Bombieri norm as well as among sublevel sets of sum of squares forms whose Gram matrix has bounded Frobenius or nuclear (or, more generally, $p$-Schatten) norm. These 
volume-minimizing properties of the Euclidean ball with respect to its representation (as a sublevel set of 
a form of fixed even degree) complement its numerous intrinsic geometric properties. We also provide a probabilistic interpretation of the results.
\end{changemargin}
\vspace{1cm}

\section*{Introduction}

It is well-known that the unit Euclidean ball $\textrm{B}_n = \{x\in \R^n: \sum_{i=1}^n x_i^2\leq 1\}$ 
has numerous (intrinsic) geometric properties. For example, $\mathrm{B}_n$ has the smallest surface area among all domains in $\R^n$ of a given volume or, equivalently, it has the largest volume among all domains of a given surface area. Hilbert and Cohn-Vossen \cite{Hilbert-52} describe ten more geometric properties of $\mathrm{B}_n$ or of its boundary $\partial\textrm{B}_n=\{x\in \R^n:\sum_{i=1}^n x_i^2=1\}$, the Euclidean sphere. In \cite{parsimony} it
was shown that $\textrm{B}_n$ exhibits some interesting extremal properties relative to its representation as a sublevel set of a nonnegative form.

More generally, in \cite{parsimony} the author was interested in properties of 
$n$-variate forms $f$ of a given degree whose sublevel set $\{f\leq 1\}=\{x\in\R^n: f(x)\leq1\}$ has fixed Lebesgue volume.
For instance, it was proved that the form $x\in \R^n \mapsto f^\star(x)=\sum_{i=1}^n x_i^{2d}$ minimizes the sparsity-inducing $\ell_1$-norm of coefficients among all $n$-variate forms of degree $2d$ whose sublevel set has the same Lebesgue volume as $\{f^\star\leq 1\}$, the unit $L^{2d}$-ball in $\R^n$. Equivalently, by homogeneity,  $f^\star$ minimizes $\mathrm{vol}\{f\leq 1\}$, the Lebesgue volume of the sublevel set,
among all $n$-variate forms $f$ of even degree $2d$ with bounded $\ell_1$-norm.

Similarly, it was proved that the form $x\in \R^n\mapsto b_{2d,n}(x)=(\sum_{i=1}^n x_i^2)^d$, whose sublevel set $\{b_{2d,n}\leq 1\}=\mathrm{B}_n$ is the unit Euclidean ball, minimizes $\mathrm{vol}\{f\leq 1\}$ among all $n$-variate forms $f$ of degree $2d$ with bounded Bombieri norm when $d=1,2,3$ and $4$. In addition, for some values of $d$, the form $b_{2d,n}$ also minimizes $\mathrm{vol}\{f\leq 1\}$ among all $n$-variate sum of squares forms $f$ of degree $2d$ whose Gram matrix has bounded trace.

Hence, the abovementioned results from \cite{parsimony} suggest that the Euclidean ball has 
volume-minimizing properties with regard to its representation as the sublevel set of
a form of fixed even degree $d$, when considering nonnegative forms of degree $d$ with bounded Bombieri norm or sum of squares forms of degree $d$ with Gram matrix of bounded trace.

\subsection*{Contribution}
This paper shows that indeed these results for the unit Euclidean ball $\mathrm{B}_{n}$ are  true for all even degrees $d$ and not only for the special cases considered in \cite{parsimony}. In fact we prove a more general result. The unit Euclidean ball $\mathrm{B}_n$ minimizes $\mathrm{vol}\{f\leq 1\}$:

- over {\em all} nonnegative $n$-variate forms $f$ of fixed (arbitrary) even degree $d$ with bounded norm, when the norm is invariant under orthogonal changes of variables, which includes {\em Bombieri norm} as 
important special case;

- over {\em all} sum of squares $n$-variate forms $f$ of fixed (arbitrary) even degree $d$, whose Gram matrix 
has bounded norm, when the norm is invariant under conjugation by orthogonal matrices.
This includes {\em Schatten $p$-norms} and, in particular, \emph{nuclear} and \emph{Frobenius} norms.

These new volume-minimizing properties of the Euclidean ball are attached to its
representation as a sublevel set of a form and complement its intrinsic geometric properties. 

Our results admit a probabilistic interpretation. The  Gaussian-like probability measure with density $x\mapsto \exp(-\kappa\,\vert x\vert^d)$ 
minimizes an $O(n)$-invariant norm $\Vert f\Vert$ over all probability measures with density $x\mapsto \exp(-f(x))$, 
where $f$ is a nonnegative form of degree $d$.

\section{Main results}
In the following we denote by $\mathcal{F}_{d,n}$ the space of $n$-ary real forms (real homogeneous polynomials) of degree $d$. 
For any form $f\in \mathcal{F}_{d,n}$ let $\{f\leq 1\}=\{x\in \R^n: f(x)\leq 1\}$ be its sublevel set at level one and let $v(f)$ denote the Lebesgue volume of $\{f\leq 1\}$,
\begin{align}
  v(f)= \mathrm{vol}\{f\leq 1\}.
\end{align}
If for $f\in \mathcal{F}_{d,n}$ the volume $v(f)$ of the sublevel set is finite, then $f$ is necessarily \emph{nonnegative}, that is, $f(x)\geq 0$ for all $x\in \R^n$. In particular, the degree $d$ must be even which we implicitly assume in the sequel.

The \emph{volume function} $v: \mathcal{F}_{d,n}\rightarrow \R_{\geq 0}\cup\{+\infty\}$ is lower-semicontinuous and homogeneous of degree $-n/d$. Moreover, forms $f\in \mathcal{F}_{d,n}$ with finite $v(f)$ constitute a convex subcone $\mathcal{V}_{d,n}$ of the cone of nonnegative forms in $\mathcal{F}_{d,n}$\footnote{Note that $\mathcal{V}_{d,n}$ does not contain the origin.} and the restriction $v|_{\mathcal{V}_{d,n}}: \mathcal{V}_{d,n}\rightarrow \R_{\geq 0}$ is strictly convex. We refer to \cite[Thm. $2.2$]{parsimony} for these results.

Let $\Vert\cdot \Vert: \mathcal{F}_{d,n}\rightarrow \R$ be any norm and consider the following convex optimization problem
\begin{align}\label{opt}
\mathbf{P_{\Vert \cdot\Vert}}:\quad\mathrm{opt}_{\Vert \cdot\Vert} = \inf\{v(f):\ \Vert f\Vert\leq 1,\ f\in \mathcal{F}_{d,n}\}.
\end{align}
\begin{remark}
Note that $\mathbf{P_{\Vert\cdot\Vert}}$ is the problem of minimization of the volume of the sublevel set $\{f\leq 1\}$ of a form $f\in \mathcal{F}_{d,n}$ over the unit ball in $\mathcal{F}_{d,n}$ defined by the norm $\Vert\cdot\Vert$. 
\end{remark}
Consider the following standard action of the group $O(n)=\{\rho\in \R^{n\times n}: \rho\rho^t=\mathrm{id}\}$ of orthogonal transformations on forms $\mathcal{F}_{d,n}$:
\begin{align}\label{action}
  \rho\in O(n), f\in \mathcal{F}_{d,n}\ \mapsto\ \rho^*f\in \mathcal{F}_{d,n},\ \rho^*f(x)=f(\rho^{-1}x).
\end{align}
A norm $\Vert\cdot\Vert : \mathcal{F}_{d,n}\rightarrow \R$ is \emph{$O(n)$-invariant} if $\Vert \rho^*f\Vert = \Vert f\Vert$ for all $\rho\in O(n)$ and $f\in \mathcal{F}_{d,n}$.

In the following theorem we show that $\mathbf{P}_{\Vert\cdot\Vert}$ has a unique optimal solution and we find it explicitly in the case of an $O(n)$-invariant norm.
\begin{theorem}\label{Th1}
  Let $d$ be even and $\Vert\cdot\Vert:\mathcal{F}_{d,n}\rightarrow \R$ be a norm. Then
  \begin{itemize}
  \item the convex optimization problem $\mathbf{P_{\Vert\cdot\Vert}}$ has a unique optimal solution $f^\star\in \mathcal{V}_{d,n}$.
  \item   If the norm $\Vert\cdot\Vert$ is $O(n)$-invariant, then $f^\star= b_{d,n}/\Vert b_{d,n}\Vert$, where\\ $b_{d,n}(x)=\vert x\vert^{d} = (x_1^2+\dots+x_n^2)^{d/2}$, and $\mathrm{opt}_{\Vert \cdot\Vert} = \Vert b_{d,n}\Vert^{n/d}v(b_{d,n})$.
    \end{itemize}
 \end{theorem}
 The first claim follows from the fact that the volume function
 is lower-semicontinuous and strictly convex, see \cite[Section $7.2$]{parsimony}. The $O(n)$-invariance of the norm and of the volume function combined with the uniqueness of the optimal solution imply the second claim. We refer to Section \ref{sec:main}, where Theorem \ref{Th1} is proved in detail.  
 
The sublevel set of $b_{d,n}$ is the unit Euclidean ball $\mathrm{B}_n=\{\vert x\vert \leq 1\}$, it does not depend on $d$ and its volume equals
\begin{align}\label{v(b)}
v(b_{d,n})=\mathrm{vol}(\mathrm{B}_n)=  \frac{\sqrt{\pi}^n}{\Gamma\left(\frac{n}{2}+1\right)}.
\end{align}
\begin{remark}
\cref{Th1} implies that the Euclidean ball in $\R^n$ of radius $\Vert b_{d,n}\Vert^{1/d}$ has smallest volume among sublevel sets of forms in the unit ball $\{f\in \mathcal{F}_{d,n}: \Vert f\Vert\leq 1\}$ in $\mathcal{F}_{d,n}$ defined by an $O(n)$-invariant norm. 
\end{remark}
Observe that if the norm $\Vert\cdot \Vert$ is not $O(n)$-invariant, then $f^\star\neq b_{d,n}/\Vert b_{d,n}\Vert$ in general. For example, when $\Vert \cdot \Vert$ is the $\ell_1$-norm of coefficients of a form written in the basis $\{x^\alpha\}_{\vert\alpha\vert=d}$ of monomials, then \cite[Thm. $3.2$]{parsimony} implies that $f^{\star}=\frac{1}{\sqrt{n}}\left( x_1^d+\cdots+x_n^d\right)$, a form not proportional to $b_{d,n}$ for $d>2$.

\begin{remark}
An interesting extension of Theorem \ref{Th1} pointed out by a referee is to consider the general setting of continuous $\lambda$-homogeneous functions on $\R^n$, where $\lambda$ is a positive real number. To establish such a generalization one need to investigate (i) continuity properties of the volume function on an appropriate infinite-dimensional (reflexive) Banach space that contains such functions,  and (ii) whether homogenity is preserved when passing to weak-limits of sequences. We leave this question for future research.
\end{remark}

Now we compute the optimal value of $\mathbf{P_{\Vert\cdot\Vert}}$ for some relevant $O(n)$-invariant norms, in view of \cref{Th1} and \cref{v(b)} this task reduces to computing $\Vert b_{d,n}\Vert$.

\subsection{Bombieri norm}
Recall first that any $f\in \mathcal{F}_{d,n}$ can be written in \emph{the basis of rescaled monomials}, 
\begin{align}\label{expansion}
f(x)=\sum_{|\alpha|=d}f_{\alpha}\sqrt{{d\choose \alpha}} x^\alpha,\ x\in \R^n,
\end{align}
where ${d\choose \alpha}=\frac{d!}{\alpha_1!\dots \alpha_n!}$ is the multinomial coefficient. \emph{The Bombieri norm of $f$} is defined as
\begin{align}\label{Bomb}
  \Vert f\Vert_B^2 = \sum_{|\alpha|=d} f_{\alpha}^2.
\end{align}
Under different names this norm appears in real algebraic geometry \cite{Reznick1992},
in perturbation theory of roots of univariate polynomials \cite{TBS2017}, in the truncated moment problem \cite{Sch2017}, in the study of random polynomials \cite{ShSm1993, FLL2015}, in the theory of symmetric tensor decompositions \cite{BCMT2010} and in many others branches of mathematics. It is well-known that Bombieri norm is $O(n)$-invariant (see, e.g., \cite[Sec. $2.1$]{AKU2019}).
\begin{corollary}[Bombieri norm]\label{cor:Bomb}
For any $f\in \mathcal{F}_{d,n}$ with $\Vert f\Vert_B\leq 1$
  \begin{align}\label{opt:Bomb}
    v\left(\frac{b_{d,n}}{\Vert b_{d,n}\Vert_B}\right) =
\left(    \prod_{i=0}^{d/2-1} \frac{2i+n}{2i+1}
    \right)^{n/2d}\frac{\sqrt{\pi}^n}{\Gamma\left(\frac{n}{2}+1\right)}\leq v(f)
\end{align}
and equality holds  if and only if $f=b_{d,n}/\Vert b_{d,n}\Vert_B$.
\end{corollary}
The second author of the present paper conjectured in \cite[p. $249$]{parsimony} the result of \cref{cor:Bomb} and proved it for any $n$ and $d=2, 4, 6$ and $8$.

\subsection{$L^p$-norms on $\mathbb{S}^{n-1}$} The following class of norms plays a fundamental role in the study of boundary value problems for partial differential equations (see, for example, \cite{Agmon1959}). Let $p\geq 1$ and define $L^p$-norm on the unit sphere
$\mathbb{S}^{n-1}=\{x\in\R^n:\vert x\vert=1\}$ as
  \begin{align}\label{L^p}
    \Vert f\Vert_{L^p(\mathbb{S}^{n-1})} = \left(\int_{\mathbb{S}^{n-1}} |f(x)|^p\,d\mathbb{S}^{n-1}\right)^{1/p},\ f\in \mathcal{F}_{d,n},  
  \end{align}
  where $d\mathbb{S}^{n-1}$ is the Riemannian volume density on $\mathbb{S}^{n-1}$. The integral in \cref{L^p} is convergent for any $f\in \mathcal{F}_{d,n}$ and the norms $\Vert\cdot\Vert_{L^p(\mathbb{S}^{n-1})}$, $p\geq 1$, are obviously $O(n)$-invariant.
  \begin{corollary}[$L^p$-norm on $\mathbb{S}^{n-1}$]\label{cor:L^p}
    For any $f\in \mathcal{F}_{d,n}$ with $\Vert f\Vert_{L^p(\mathbb{S}^{n-1})}\leq 1$
    \begin{align}\label{opt:L^p}
v\left(\frac{b_{d,n}}{\Vert b_{d,n}\Vert_{L^p(\mathbb{S}^{n-1})}}\right) = \left(\frac{2\sqrt{\pi}^{n}}{\Gamma\left(\frac{n}{2}\right)}\right)^{n/dp}\frac{\sqrt{\pi}^n}{\Gamma\left(\frac{n}{2}+1\right)}\leq v(f)
    \end{align}
    and equality holds  if and only if $f=b_{d,n}/\Vert b_{d,n}\Vert_{L^p(\mathbb{S}^{n-1})}$.
  \end{corollary}

\subsection{Uniform norm on $\mathbb{S}^{n-1}$}
As the limiting case of $L^p(\mathbb{S}^{n-1})$-norms when $p\rightarrow +\infty$ one obtains \emph{the uniform norm} on the unit sphere $\mathbb{S}^{n-1}$,
\begin{align}\label{uniform}
  \Vert f\Vert_{L^\infty(\mathbb{S}^{n-1})} = \max\limits_{x\in \mathbb{S}^{n-1}} |f(x)|.
\end{align}
\begin{corollary}[Uniform norm on $\mathbb{S}^{n-1}$]\label{cor:uniform}
For any $f\in \mathcal{F}_{d,n}$ with $\Vert f\Vert_{L^\infty(\mathbb{S}^{n-1})}\leq 1$
  \begin{align}\label{opt:uniform}
     v\left(\frac{b_{d,n}}{\Vert b_{d,n}\Vert_{L^\infty(\mathbb{S}^{n-1})}}\right) = \frac{\sqrt{\pi}^n}{\Gamma\left(\frac{n}{2}+1\right)}\leq v(f)
   \end{align}
   and equality holds  if and only if $f=b_{d,n}$.
  \end{corollary}
    Note that \cref{opt:uniform} can be considered as the limiting case of \cref{opt:L^p} when $p\rightarrow +\infty$.

\subsection{Nuclear norm} 
    \emph{Nuclear norm} appears in the study of tensor decompositions \cite{FL2016} and in the theory of rank-one approximations of tensors \cite{LNSU2018, AKU2019}. For $f\in \mathcal{F}_{d,n}$ it is defined as
    \begin{align}\label{nucl}
      \Vert f\Vert_* = \inf\left\{\sum_{k=1}^r |\lambda_k|:\ f(x)=\sum\limits_{k=1}^r \lambda_k(y^k\cdot x)^d,\ \lambda_k\in \R,\ y^k\in \mathbb{S}^{n-1}\right\},
    \end{align}
    where $(y\cdot x) = y_1x_1+\dots+y_nx_n$ denotes the dot product of two vectors in $\R^n$.
    \begin{corollary}[Nuclear norm]\label{cor:nucl}
      For any $f\in \mathcal{F}_{d,n}$ with $\Vert f\Vert_*\leq 1$
      \begin{align}
        v\left(\frac{b_{d,n}}{\Vert b_{d,n}\Vert_*}\right) = \left(    \prod_{i=0}^{d/2-1} \frac{2i+n}{2i+1}\right)^{n/d}\frac{\sqrt{\pi}^n}{\Gamma\left(\frac{n}{2}+1\right)}\leq v(f)
      \end{align}
      and equality holds  if and only if $f=b_{d,n}/\Vert b_{d,n}\Vert_*$.
    \end{corollary}    
    A form $f\in \mathcal{F}_{d,n}$ of even degree $d$ is called \emph{a sum of squares} if $f=s_1^2+\dots+s_r^2$ for some forms $s_1,\dots,s_r\in \mathcal{F}_{d/2,n}$ of  degree $d/2$. Any sum of squares form is non-negative. Fix a total order $\leq$ on the set $\left\{\sqrt{{d/2\choose \alpha}}x^\alpha: \vert\alpha\vert=d/2\right\}$ of rescaled monomials of degree $d/2$ (e.g., the lexicographic order) and denote by $N = {d/2+n-1\choose n-1}$ the dimension of $\mathcal{F}_{d/2,n}$. Then, $f\in \mathcal{F}_{d,n}$ is a sum of squares if and only if there exists a positive semidefinite real symmetric matrix $G\in \mathcal{S}_N$, called \emph{Gram matrix}, satisfying
    \begin{align}\label{Gram}
      f(x) = m_{d/2}(x)^t  G\, m_{d/2}(x),\ x\in \R^n,
    \end{align}
    where $m_{d/2}(x)$ denotes the $N$-dimensional column-vector of rescaled monomials $\sqrt{{d/2\choose\alpha}}x^\alpha$, $\vert \alpha\vert=d/2$, ordered with respect to $\leq$ (see \cite[$\S 2$]{CLR1994} and \cref{lem:Gram}). Note that the cone of sums of squares in $\mathcal{F}_{d,n}$ is the image of the closed convex cone $\mathcal{PSD}_N\subset \mathcal{S}_N$  of positive semidefinite matrices under linear map \eqref{Gram}. 
    
    Fix a norm $\Vert \cdot\Vert$ on the space $\mathcal{S}_N$ of real symmetric $N\times N$ matrices and consider the following optimization problem: 
\begin{align}\label{SOS}
\mathbf{P}_{\Vert\cdot\Vert}^{\mathrm{sos}}:\quad\mathrm{opt}^{\mathrm{sos}}_{\Vert\cdot\Vert} = \inf\{v(f):\ f=m_{d/2}(x)^t G\, m_{d/2}(x),\ G\in \mathcal{PSD}_N,\ \Vert G\Vert \leq 1\}.
\end{align}
\begin{remark}
  Note that $\mathbf{P}^{\mathrm{sos}}_{\Vert\cdot\Vert}$ is the problem of minimization of the volume of the sublevel set $\{f\leq 1\}$ of a sum of squares $f=m_{d/2}(x)^tG\,m_{d/2}(x)$ with Gram matrix $G$ from the unit ball in $\mathcal{S}_N$ defined by the norm $\Vert\cdot\Vert$.
\end{remark}
A norm $\Vert \cdot\Vert: \mathcal{S}_N \rightarrow \R$ is said to be \emph{$O(N)$-invariant} if $\Vert R^tGR\Vert =\Vert G\Vert$ for all $R\in O(N)$ and $G\in \mathcal{S}_N$.
We prove that problem $\mathbf{P}_{\Vert\cdot\Vert}^\mathrm{sos}$ has a unique optimal solution, which, in the case of an $O(N)$-invariant norm, is proportional to $b_{d,n}$.
\begin{theorem}\label{Th2}
  Let $d$ be even and $\Vert \cdot\Vert:\mathcal{S}_N \rightarrow\R$ be a norm. Then
  \begin{itemize}
  \item $\mathbf{P}_{\Vert \cdot \Vert}^{\mathrm{sos}}$ is a convex optimization problem with a unique optimal solution $f^\star_{\mathrm{sos}}$.
\item If norm $\Vert \cdot\Vert$ is $O(N)$-invariant, then $f^\star_{\mathrm{sos}} = b_{d,n}/\Vert \mathrm{id}_N\Vert$, where $\mathrm{id}_N\in \mathcal{S}_N$ is the identity matrix, and $\mathrm{opt}_{\Vert\cdot\Vert}^{\mathrm{sos}} = \Vert \mathrm{id}_N\Vert^{n/d}v(b_{d,n})$.    
  \end{itemize}
\end{theorem}
The first claim follows from convexity properties of the norm, the cone of positive semidefinite matrices and the volume function. Existence and uniqueness of an optimal solution is derived from the fact that the volume function is lower-semicontinuous and strictly convex and from the fact that the map \eqref{Gram} sending a real symmetric matrix to a real form is linear. The $O(N)$-invariance of the norm and of the volume function combined with the uniqueness of the optimal solution imply the last claim. A more detailed proof of Therem \ref{Th2} is given in Section \ref{sec:main}.

\begin{remark}
  \cref{Th2} implies that the Euclidean ball in $\R^n$ of radius $\Vert \mathrm{id}_N\Vert^{1/d}$ has smallest volume among sublevel sets of sums of squares corresponding to Gram matrices from the unit ball $\{G\in \mathcal{S}_N: \Vert G\Vert\leq 1\}$ in $\mathcal{S}_N$ defined by an $O(N)$-invariant norm.  
\end{remark}
\subsection{Schatten $p$-norms} Given a real symmetric matrix $G\in \mathcal{S}_N$ its \emph{Schatten $p$-norm}, $p\geq 1$, is defined by
\begin{align}\label{Schatten}
  \Vert G\Vert_p = \left(\sum_{i=1}^N \vert \lambda_i(G)\vert^p\right)^{1/p},
\end{align}
where $\lambda_1(G), \dots, \lambda_N(G)\in \R$ are the eigenvalues of $G$. When $p=1$ this norm is also known as \emph{nuclear norm} and if $p=2$ we recover \emph{Frobenius norm} which is classically used in the context of low-rank approximation of matrices \cite{EY1936}. Since eigenvalues do not change under conjugation by orthogonal matrices, all Schatten $p$-norms are $O(N)$-invariant.

We next compute the optimal value of problem $\mathbf{P}_{\Vert\cdot\Vert}^{\mathrm{sos}}$ for Schatten $p$-norms. 
Again, as in the above case of general nonnegative forms, by Theorem \ref{Th2} this task reduces to computing the norm of $\mathrm{id}_N\in \mathcal{S}_N$. 
\begin{corollary}[Schatten $p$-norms]\label{cor:Schatten}
Let $p\geq 1$. Then for any sum of squares form $f=m_{d/2}(x)^tG\,m_{d/2}(x)\in \mathcal{F}_{d,n}$, $G\in \mathcal{PSD}_N$, with $\Vert G\Vert_p\leq 1$
  \begin{align}
    v\left( \frac{b_{d,n}}{\Vert \mathrm{id}_N\Vert_p}\right) = N^{n/dp}\frac{\sqrt{\pi}^n}{\Gamma\left(\frac{n}{2}+1\right)} \leq v(f)
  \end{align}
  and equality holds  if and only if $f=b_{d,n}/N^{1/p}$.
\end{corollary}
\begin{remark}  
In \cite{parsimony} the second author of the present paper considered an analogous problem to $\mathbf{P}_{\Vert\cdot\Vert_2}^{\mathrm{sos}}$, where $m_{d/2}(x)$ is replaced by the vector of monomials $x^\alpha$, $\vert\alpha\vert=d/2$, (without coefficients $\sqrt{{d/2\choose \alpha}}$), and proved that $b_{d,n}$ is (up to a multiple) a unique optimal solution when $d=2, 4$ and when $d\in 4\mathbb{N}$ provided that $n$ is large  \cite[Thm. $5.1$]{parsimony}.
\end{remark}

\cref{cor:Schatten} immediately follows from \cref{Th2}, definition of Schatten $p$-norms \eqref{Schatten} and formula \eqref{v(b)} for the volume of the sublevel set of $b_{d,n}$.

\subsection{Spectral norm} The 
\emph{Spectral norm} of $G\in \mathcal{S}_N$ defined by
\begin{align}
  \Vert G\Vert_{\sigma} = \max_{i=1,\dots,N}\vert\lambda_i(G)\vert
\end{align}
can be considered as the limit of Schatten $p$-norms \eqref{Schatten} as $p\rightarrow+\infty$.
\begin{corollary}[Spectral norm]
  For any sum of squares $f=m_{d/2}(x)^tG\,m_{d/2}(x)\in \mathcal{F}_{d,n}$, $G\in \mathcal{PSD}_N$, with $\Vert G\Vert_\sigma\leq 1$
  \begin{align}
      v\left( \frac{b_{d,n}}{\Vert \mathrm{id}_N\Vert_\sigma}\right) =\frac{\sqrt{\pi}^n}{\Gamma\left(\frac{n}{2}+1\right)} \leq v(f)
  \end{align}
  and equality holds  if and only if $f=b_{d,n}$.
\end{corollary}


\subsection{Probabilistic interpretation of results}

If for $f\in \mathcal{V}_{d,n}$ the sublevel set $\{f\leq1\}$ has finite Lebesgue volume, then by \cite[Thm. $2.2$]{lass-homog}
\begin{equation}\label{Laplace}
v(f)\,=\,\frac{1}{\Gamma(1+n/d)}\,\int_{\R^n}\exp(-f(x))\,dx,
\end{equation}
see also \cite{morozov}. When $\int_{\R^n}\exp(-f(x))\,dx=1$ the function
$x\mapsto \exp(-f(x))$ is the density of a probability measure $\mu_f$ on $\R^n$. In particular, if
$f^*(x)=\kappa\,\vert x\vert^d$ with
\begin{equation}
\label{kappa}
\kappa\,=\,\left(\frac{\Gamma(1+n/d)}{\Gamma(1+n/2)}\right)^{d/n}\pi^{d/2},\end{equation}
then $\mu_{f^*}$ is a Gaussian-like probability measure in the sense that all of its moments are easily obtained from those of a Gaussian measure, that is,
\begin{align}
\int_{\R^n}x^\alpha\,\exp(-f^*(x))\,dx\,=\,\frac{\Gamma(1+(n+\vert\alpha\vert)/d)}{\Gamma(1+(n+\vert\alpha\vert)/2)}
\int_{\R^n}x^\alpha\,\exp(-\kappa^{2/d}\,\vert x\vert^2)\,dx,\quad\forall\alpha\in\N^n.
\end{align}
By \cref{Laplace} and homogeneity, $f$ is the unique optimal solution of $\mathbf{P}_{\Vert\cdot\Vert}$ if and only if
$\left(\mathrm{opt}_{\Vert\cdot\Vert}\Gamma(1+n/d)\right)^{d/n} f$ is the unique optimal solution of the convex optimization problem
\begin{align}\label{equiv}
\mathbf{P}^*_{\Vert \cdot\Vert}:\quad\mathrm{opt}_{\Vert \cdot\Vert}^* = \inf\left\{\:\Vert f\Vert : \int_{\R^n} \exp(-f(x))\,dx\leq 1,\ f\in \mathcal{F}_{d,n}\right\}.
\end{align}
In light of this fact \cref{Th1} can be equivalently stated as follows.
\begin{theorem}\label{Th1-prime}
  Let $d$ be even and $\Vert\cdot\Vert:\mathcal{F}_{d,n}\rightarrow \R$ be a norm. Then
  \begin{itemize}
  \item the convex optimization problem $\mathbf{P}^*_{\Vert\cdot\Vert}$ has a unique optimal solution $f^\star\in \mathcal{V}_{d,n}$.
  \item If the norm $\Vert\cdot\Vert$ is $O(n)$-invariant, then $f^\star= \kappa\,b_{d,n}$, where\\ $b_{d,n}(x)=\vert x\vert^{d}$, $\kappa$ is as in \eqref{kappa}, and $\mathrm{opt}^*_{\Vert \cdot\Vert} =\kappa\, \Vert b_{d,n}\Vert$.
    \end{itemize}
\end{theorem}
\begin{remark}
If $\Vert\cdot\Vert$ is $O(n)$-invariant, then \cref{Th1-prime} implies that the Gaussian-like probability density $x\mapsto \exp(-\kappa\,\vert x\vert^d)$ minimizes $\Vert f\Vert$ over all probability measures $\mu_f$ with density $x\mapsto \exp(-f(x))$, where $f\in\mathcal{F}_{d,n}$ is a nonnegative form 
of degree $d$.
\end{remark}

\section{Preliminaries and auxiliary results}\label{sec:prel}
In this section we give necessary definitions and prove some auxiliary results that are needed in \cref{sec:main}.

Recall that $\mathcal{F}_{d,n}$ denotes the space of $n$-ary real forms (or homogeneous polynomials) of degree $d$ endowed with a norm $\Vert\cdot\Vert: \mathcal{F}_{d,n}\rightarrow \R$. Furthermore, recall that the group \\$O(n)=\{\rho\in \R^{n\times n}: \rho\rho^t=\mathrm{id}\}$ of orthogonal transformations acts on $\mathcal{F}_{d,n}$ as follows
\begin{align}\label{action2}
    \rho\in O(n), f\in \mathcal{F}_{d,n}\ \mapsto\ \rho^*f\in \mathcal{F}_{d,n},\ \rho^*f(x)=f(\rho^{-1}x).
\end{align}
For even $d$ the form
  \begin{align}\label{b-expansion}
    b_{d,n}(x) = (x_1^2+\dots+x_n^2)^{d/2} = \sum_{\vert \beta\vert=d/2} {d/2 \choose \beta} x_1^{2\beta_1}\dots x_n^{2\beta_n},\ x\in \R^n,
  \end{align}
is obviously invariant with respect to \eqref{action2}.
The following easy lemma asserts that $b_{d,n}$ is essentially the only invariant form.
\begin{lemma}\label{invariance}
Let $f\in \mathcal{F}_{d,n}$ be a non-zero form invariant under $O(n)$-action \cref{action}. Then $d$ is even and $f$ is proportional to $b_{d,n}$.
\end{lemma}
\begin{proof}
$O(n)$-invariance of $f$ implies $f(x)=c$ whenever $\vert x\vert=1$, for some constant $c$.   If the degree $d$ is odd, we have $f(-x)=-f(x)$ for any $x\in\R^n$ and hence $f=0$. Thus $d$ must be even and by homogeneity of $f$
  \begin{align}
    f(x)=f\left(\vert x\vert \frac{x}{\vert x\vert}\right) =c\vert x\vert^d = c\,b_{d,n}(x)\quad \textrm{for any}\quad x\neq 0.
  \end{align} 
\end{proof}

For two real forms $f, g\in \mathcal{F}_{d,n}$ define their \emph{Bombieri product} as
\begin{align}\label{Bombieri}
  \langle f,g\rangle_B = \sum_{\vert\alpha\vert=d} f_\alpha \,g_\alpha, 
\end{align}
where $\{f_\alpha\}_{\vert\alpha\vert=d}$ and $\{g_\alpha\}_{\vert\alpha\vert=d}$ are the coefficients of $f$ and $g$ in the basis of rescaled monomials \eqref{expansion}. Equivalently, denoting by $f(\partial)$ the differential opearator obtained from $f$ by replacing variable $x_i$, $i=1,\dots,n$, with 
partial derivative $\partial/\partial x_i$,  one can show that
\begin{align}\label{eq:eval}
  \langle f,g\rangle_B = \frac{1}{d!}f(\partial)g(x).
\end{align}
From this, taking Bombieri product with a power of a linear form  
$x\mapsto f(x)=(y\cdot x)^d$, $y\in \R^n$, amounts to evaluation at $y$, that is,
\begin{align}\label{eq:eval}
  \langle f,g\rangle_B = g(y).
\end{align}

Recall that a form $f\in \mathcal{F}_{d,n}$ of even degree $d$ is called a \emph{sum of squares} if $f=s_1^2+\dots+s_r^2$ for some $s_1,\dots,s_r\in \mathcal{F}_{d/2,n}$. The following characterization of sums of squares is well-known;  we state it here as our version concerns rescaled monomials, cf. \cite[$\S 2$]{CLR1994}.
\begin{lemma}\label{lem:Gram}
  A form $f\in \mathcal{F}_{d,n}$ is a sum of squares if and only if there exists a positive semidefinite real symmetric matrix $G\in \mathcal{PSD}_N$ such that
  \begin{align}\label{eq:Gram}
    f(x)= m_{d/2}(x)^t G\, m_{d/2}(x),\ x\in \R^n,
  \end{align}
  where $N=\dim \mathcal{F}_{d/2,n}={d/2+n-1\choose n-1}$ and $m_{d/2}(x)$ is the column-vector of rescaled monomials $\sqrt{{d/2\choose \alpha}}x^\alpha$, $\vert \alpha\vert=d/2$, ordered with respect to a fixed order $\leq$.
\end{lemma}
\begin{proof}
  If $f=s_1^2+\dots+s_r^2$, then $f=m_{d/2}(x)^tG \,m_{d/2}(x)$, where $G=\sum_{i=1}^r\vec{s}_i\vec{s}_i^{\,t}\in \mathcal{PSD}_N$ and $\vec{s}_i$ denotes the column-vector of coefficients of form $s_i$, $i=1,\dots, r$, in the basis of rescaled monomials ordered with respect to $\leq$. Conversely, if  \cref{eq:Gram} holds for some positive semidefinite matrix $G=G^{1/2}G^{1/2}\in\mathcal{PSD}_N$, then $f=s_1^2+\dots+s_N^2$, where $s_1,\dots, s_N\in\mathcal{F}_{d/2,n}$ are the entries of the $N$-dimensional vector of forms $G^{1/2}m_{d/2}(x)$.
\end{proof}
\section{Proof of main results}\label{sec:main}

In this section we prove our main results, \cref{Th1} and \cref{Th2}.

\begin{proofof}{\cref{Th1}}
  The proof of the fact that $\mathbf{P}_{\Vert\cdot\Vert}$ has a unique optimal solution is analogous to the one of \cite[Thm. $3.2$]{parsimony}; we give it here for the sake of completeness.

Let $\{f_k\}_{k\in \mathbb{N}}$ be a minimizing sequence of the optimization problem $\mathbf{P}_{\Vert\cdot\Vert}$, i.e., $\Vert f_k\Vert \leq 1$, $k\in \mathbb{N}$, and $\lim_{k\rightarrow +\infty}v(f_k) = \mathrm{opt}_{\Vert\cdot\Vert}$. By compactness of the unit ball of norm $\Vert\cdot\Vert$ there is a subsequence $\{f_{k_m}\}_{m\in \mathbb{N}}$ and $f^\star\in \mathcal{F}_{d,n}$ such that $\lim_{m\rightarrow+\infty} \Vert f_{k_m}- f^\star\Vert =0$ and $\Vert f^\star\Vert\leq 1$, meaning that $f^\star$ is feasible.
By \cite[Lemma $2.3$]{parsimony}, the function $v: \mathcal{F}_{d,n}\rightarrow \R_{\geq 0}\cup\{+\infty\}$ is lower-semicontinuous. This implies 
\begin{align}
 \mathrm{opt}_{\Vert \cdot\Vert} = \liminf_{m\rightarrow +\infty} v(f_{k_m}) \geq v(f^\star),
\end{align}
that is, $f^\star$ is an optimal solution of $\mathbf{P}_{\Vert\cdot\Vert}$.

Now, by \cite[Thm. $2.2$]{parsimony}, the function $v$ is strictly convex. Thus, if $\mathbf{P}_{\Vert\cdot\Vert}$ had two different optimal solutions $f_1^\star$ and $f_2^\star$, then for $\alpha\in (0,1)$ we would have

\begin{equation}
  \begin{aligned}
  \Vert \alpha f^\star_1+(1-\alpha)f^\star_2\Vert &\leq \alpha\Vert f^\star_1\Vert+(1-\alpha)\Vert f_2^\star\Vert \leq 1,\\
  v(\alpha f_1^\star+(1-\alpha)f^\star_2)&<\alpha v(f^\star_1)+(1-\alpha)v(f^\star_2) = \mathrm{opt}_{\Vert \cdot\Vert},
\end{aligned}
\end{equation}
a contradiction. Thus an optimal solution of $\mathbf{P}_{\Vert\cdot\Vert}$ is unique. Moreover, homogeneity of $\Vert\cdot\Vert$ and $v$ implies that the unique optimal solution $f^\star$ of $\mathbf{P}_{\Vert\cdot\Vert}$ must satisfy $\Vert f^\star\Vert=1$.

Let us now consider the case of an $O(n)$-invariant norm. Observe first that the volume function $v$ is $O(n)$-invariant, that is, $v(\rho^*f)=v(f)$ for any $f\in \mathcal{F}_{d,n}$ and $\rho\in O(n)$. Indeed, this follows directly from the definition of $v(f)$ and invariance of Lebesgue measure on $\R^n$. We claim that the optimal solution $f^\star$ of $\mathbf{P}_{\Vert\cdot\Vert}$ is $O(n)$-invariant. If not there exists $\rho\in O(n)$ such that $\rho^*f^\star\neq f^\star$. Then in view of $O(n)$-invariance of $v$ and $\Vert\cdot\Vert$, $f^\star$ and $\rho^*f^\star$ are two different optimal solutions of $\mathbf{P}_{\Vert\cdot\Vert}$, which is impossible by the above. \cref{invariance} implies that $f^\star$ is proportional to $b_{d,n}$ and since $\Vert f^{\star}\Vert=1$ we must have $f^\star = b_{d,n}/\Vert b_{d,n}\Vert$. As $v$ is homogeneous of degree $-n/d$, we obtain $\mathrm{opt}_{\Vert \cdot\Vert} =v(b_{d,n}/\Vert b_{d,n}\Vert) = \Vert b_{d,n}\Vert^{n/d} v(b_{d,n})$.

\end{proofof}
If $\Vert\cdot\Vert$ is a particular norm then by \cref{Th1}, computing the optimal value of $\mathbf{P}_{\Vert\cdot\Vert}$ reduces to computing $\Vert b_{d,n}\Vert$. We next evaluate $\Vert b_{d,n}\Vert$ 
for Bombieri norm, $L^p(\mathbb{S}^{n-1})$-norm, uniform norm on $\mathbb{S}^{n-1}$, nuclear norm, and thus prove Corollaries \ref{cor:Bomb}, \ref{cor:L^p}, \ref{cor:uniform} and \ref{cor:nucl}.
\begin{proofof}{\cref{cor:Bomb}}
  By \cite[$(8.19)$]{Reznick1992} we have
  \begin{align}\label{Bomb(b)}
    \Vert b_{d,n}\Vert_B = \sqrt{    \prod_{i=0}^{d/2-1} \frac{2i+n}{2i+1}}.
  \end{align}
Combining this formula with \eqref{v(b)} yields \eqref{opt:Bomb}.
\end{proofof}
\begin{proofof}{\cref{cor:L^p}}
  One has
  \begin{align}
    \Vert b_{d,n}\Vert_{L^p(\mathbb{S}^{n-1})} = \left(\int_{\mathbb{S}^{n-1}} |b_{d,n}(x)|^p\,d\mathbb{S}^{n-1}(x)\right)^{1/p} = \vol(\mathbb{S}^{n-1})^{1/p} = \left(\frac{2\sqrt{\pi}^n}{\Gamma\left(\frac{n}{2}\right)}\right)^{1/p},
  \end{align}
   which together with \cref{Th1} and \eqref{v(b)} yields \eqref{opt:L^p}.
 \end{proofof}
\cref{cor:uniform} follows from \cref{Th1}, \eqref{v(b)} and $\Vert b_{d,n}\Vert_{L^{\infty}(\mathbb{S}^{n-1})} = \max_{x\in \mathbb{S}^{n-1}}\vert x\vert^{d} = 1$. 
\begin{proofof}{\cref{cor:nucl}}
From a result of Hilbert 
\cite{Hilbert1909} it follows that there exist $r\in \mathbb{N}$, $\lambda_1,\dots, \lambda_r>0$ and $y^1,\dots,y^r\in \mathbb{S}^{n-1}$ such that
\begin{align}\label{eq:deco}
  b_{d,n}(x)=  \sum_{k=1}^r \lambda_k(y^k\cdot x)^d
\end{align}
and thus, invoking \cite[Example $1.1$]{Nie2017}, we have
\begin{align}
  \Vert b_{d,n}\Vert_* = \sum_{k=1}^r\lambda_k.
\end{align}
On the other hand, by \cref{eq:eval} and \cref{eq:deco},
\begin{align}
\Vert b_{d,n}\Vert_* = \sum_{k=1}^r\lambda_k = \sum_{k=1}^r\lambda_k \langle (y^k\cdot\bullet)^d, b_{d,n}\rangle_B=\langle b_{d,n},b_{d,n}\rangle_B = \Vert b_{d,n}\Vert^2_B,
\end{align}
and \eqref{cor:nucl} follows from \cref{Bomb(b)} and \cref{v(b)}.
\end{proofof}

We now prove \cref{Th2}.
\begin{proofof}{\cref{Th2}}
  Let us observe first that convexity of the feasible set
  \begin{align}
  \{f\in \mathcal{F}_{d,n}: f=m_{d/2}(x)^tG\, m_{d/2}(x),\ G\in \mathcal{PSD}_N,\ \Vert G\Vert\leq 1\}  
  \end{align}
of optimization problem $\mathbf{P}_{\Vert \cdot\Vert}^{\mathrm{sos}}$ follows directly from convexity of the cone $\mathcal{PSD}_N$ of positive semidefinite matrices and convexity of norm $\Vert\cdot\Vert$. 
This fact combined with convexity of the function $v$ (see \cite[Thm. $2.2$]{parsimony})  implies that $\mathbf{P}_{\Vert\cdot\Vert}^{\mathrm{sos}}$ is a convex optimization problem.

Let $\{f_k=m_{d/2}(x)^t G_k\, m_{d/2}(x)\}_{k\in \mathbb{N}}$ be a minimizing sequence of $\mathbf{P}_{\Vert\cdot\Vert}^{\mathrm{sos}}$, i.e., $G_k\in \mathcal{PSD}_N$, $\Vert G_k\Vert\leq 1$, $k\in \mathbb{N}$, and $\lim_{k\rightarrow +\infty} v(f_k) = \mathrm{opt}_{\Vert\cdot\Vert}^{\mathrm{sos}}$. Since $\{G\in \mathcal{PSD}_N: \Vert G\Vert\leq 1\}$ is compact, there is a subsequence $\{G_{k_m}\}_{m\in \mathbb{N}}$ and a matrix $G^\star\in \mathcal{PSD}_N$, $\Vert G^\star\Vert\leq 1$, such that $\lim_{m\rightarrow +\infty} \Vert G_{k_m}-G^\star\Vert =0$. In particular, the sum of squares form $f^\star = m_{d/2}(x)^t G^\star \,m_{d/2}(x)\in \mathcal{F}_{d,n}$ is feasible for $\mathbf{P}_{\Vert\cdot\Vert}^{\mathrm{sos}}$ and coefficients of $f_{k_m}$ converge 
to coefficients of $f^\star$, as $m\rightarrow+\infty$. Next, since the function  $v:\mathcal{F}_{d,n}\rightarrow \R_{\geq 0}\cup\{+\infty\}$ is lower-semicontinuous \cite[Lemma $2.3$]{parsimony} we have
\begin{align}
  \mathrm{opt}_{\Vert\cdot\Vert}^{\mathrm{sos}} = \liminf_{m\rightarrow +\infty} v(f_{k_m}) \geq v(f^\star),
\end{align}
that is, $f^\star$ is an optimal solution of $\mathbf{P}_{\Vert\cdot\Vert}^{\mathrm{sos}}$.  Exactly in the same way as in the proof of \cref{Th1}, strict convexity of $v$ implies that $f^\star$ is a unique optimal solution of $\mathbf{P}_{\Vert \cdot\Vert}^{\mathrm{sos}}$. Also, from homogeneity of $v$ and $\Vert\cdot\Vert$ we obtain $\Vert G^\star \Vert=1$.

Let now $\Vert\cdot\Vert:\mathcal{S}_N \rightarrow \R$ be an $O(N)$-invariant norm. If $f=m_{d/2}(x)^t G\,m_{d/2}(x)$ is feasible for $\mathbf{P}_{\Vert\cdot\Vert}^{\mathrm{sos}}$, i.e., $G\in \mathcal{PSD}_N$ and $\Vert G\Vert\leq 1$, then so is $\rho^*f$ for any $\rho\in O(n)$. Indeed, since Bombieri product \eqref{Bombieri} is invariant under $O(n)$-action \eqref{action} and since the rescaled monomials $\sqrt{{d/2\choose \alpha}}x^\alpha$, $\vert\alpha\vert=d/2$, form an orthonormal basis of $\mathcal{F}_{d/2,n}$ with respect to Bombieri product, for any $\rho\in O(n)$ there exists $R=R(\rho)\in O(N)$ such that
\begin{align}
\rho^*f = m_{d/2}(x)^t R^tGR\,m_{d/2}(x).
\end{align}
Hence, by $O(N)$-invariance of $\mathcal{PSD}_N$ and $\Vert\cdot\Vert$, we have $R^tGR\in \mathcal{PSD}_N$ and $\Vert R^tGR\Vert=\Vert G\Vert\leq 1$ or, in other words, $\rho^*f$ is feasible for $\mathbf{P}_{\Vert\cdot\Vert}^{\mathrm{sos}}$. Therefore the unique optimal solution $f^\star$ of $\mathbf{P}_{\Vert\cdot\Vert}^{\mathrm{sos}}$ must be $O(n)$-invariant, that is, $\rho^*f^\star=f^\star$ for all $\rho\in O(n)$. 

Next, by \cref{invariance}, $f^\star$ is proportional to $b_{d,n}$. From \eqref{b-expansion} we have that $b_{d,n} = m_{d/2}(x)^tm_{d/2}(x)$, namely the identity matrix $\mathrm{id}_N\in \mathcal{PSD}_N$ is a Gram matrix of $b_{d,n}$. Therefore $f^\star = b_{d,n}/\Vert \mathrm{id}_N\Vert$ as its Gram matrix satisfies $\Vert \mathrm{id}_N/\Vert \mathrm{id}_N\Vert\Vert =1$. Also, $\mathrm{opt}_{\Vert\cdot\Vert}^{\mathrm{sos}} = v(b_{d,n}/\Vert\mathrm{id}_N\Vert) = \Vert \mathrm{id}_N\Vert^{n/d}v(b_{d,n})$ by homogeneity of the volume function $v$.
\end{proofof}

\section{Conclusion}

We have provided new volume-minimizing properties of the Euclidean unit ball. 
In contrast to its intrinsic geometric properties, they are attached to its representation as the 
sublevel set of a form of fixed even degree. The minimum is over all nonnegative forms of same degree
with bounded norm or over sum of squares forms of same degree whose Gram matrix has bounded norm, for certain families of norms.

\section*{Acknowledgements}
We are thankful to Jiawang Nie for helping us with the proof of Corollary \ref{cor:nucl} and to anonymous referees for their useful comments and remarks.

\bibliographystyle{amsalpha}


\end{document}